\documentclass[12pt, final]{amsproc}
\usepackage{amssymb,amsbsy,amsmath,amsfonts,mathpazo}
\usepackage{latexsym,euscript,exscale}
\usepackage{mathrsfs}

\usepackage[all]{xy}
\usepackage[notcite,notref]{showkeys}

\usepackage{helvet}

  \newcounter{itemizedlistcounter}  
\newenvironment{itemized}
  {\begin{list}
     {(\arabic{itemizedlistcounter})} 
     {\usecounter{itemizedlistcounter}   
      \setlength{\leftmargin}{0.5em}} 
  }
  {\end{list}}

\author[F. Cabello S\'anchez]{F\'elix Cabello S\'anchez}
\address{Departamento de Matem\'aticas and IMUEx, Universidad de Extremadura\\
Avenida de Elvas, 06071-Badajoz, Espa\~na.
\newline
Orcid Id: 0000-0003-0924-5189}
\email{fcabello@unex.es}

\newcommand{\eps}{\varepsilon}

\newcommand {\N}{\mathbb N}

\newcommand {\R}{\mathbb R}

\renewcommand{\leq}{\ensuremath{\leqslant}}
\renewcommand{\geq}{\ensuremath{\geqslant}}

\renewcommand{\dim}{\ensuremath{\mathop{\rm dim\,}}}

\newcommand{\supp}{\operatorname{supp}}
\newcommand{\Aut}{\operatorname{Aut}}
\newcommand{\Iso}{\operatorname{Iso}}
\newcommand{\Hol}{\operatorname{Hol}}
\newcommand{\Arg}{\operatorname{Arg}}

\providecommand{\supp}{\mathop{\rm supp}\nolimits}

\newtheorem{thm}{Theorem}[section]
\newtheorem{cor}[thm]{Corollary}
\newtheorem{theorem}[thm]{Theorem}

\newtheorem{lemma}[thm]{Lemma}
\newtheorem{prop}[thm]{Proposition}

\newtheorem{fact}[thm]{Fact}

\newtheorem{noresult}[thm]{ }
\newtheorem*{claim*}{Claim}

\theoremstyle{definition}
\newtheorem{defi} [thm] {Definition}

\newtheorem{prob}[thm]{Problem}
\newtheorem*{prob*}{Problem}

\newtheorem{quest}{Question}

\newcommand{\m}[1]{\marginpar{\tiny{#1}}}

\title{Wheeling around Mazur rotations problem}

\keywords{Mazur rotations problem; transitivity; microtransitivity; semitransitivity; Banach space; Hilbert space}

\thanks{2020 {\it Mathematics Subject Classification}. 46B03, 46B04, 46C15}

\thanks{Supported in part by  PID2019-103961GB-C21 and Junta de Extremadura, Project IB16056.}

\begin{document}

\begin{abstract}
We study Mazur rotations problem focusing on the metric aspects of the action of the isometry group and semitransitivity properties.
\end{abstract}



\maketitle

\section*{Introduction}
Hilbert spaces have the following remarkable property for which they are called {\em transitive}: each point of the unit sphere can be mapped into any other through a (linear, surjective) isometry.

This quickly follows from the 2-dimensional case and the existence of orthogonal complements.

The infamous problem alluded to in the title asks if the Hilbert spaces are the only separable Banach spaces enjoying that {\em rotation} property. 
This question has a considerable pedigree and was in fact included by Banach in his {\em Th\'eorie des Op\'era\-tions Lin\'eaires}, cf. \cite[la remarque \`a la section~5 du chapitre XI]{Banach}.

One can split quite naturally 
Mazur's problem into two parts:
\smallskip

\begin{quest}Is every separable transitive Banach space isomorphic to a Hilbert space?
\end{quest}

\begin{quest}Is every transitive norm on a Hilbert space Euclidean? 
\end{quest}

While the former question has a more isomorphic flavour,
the latter clearly belongs to the isometric theory of Banach spaces.
Both questions remain open to this day, with no solution in sight.

We will not explain here what is known and what is not known about 
these problems since we already have \cite{becerra,CFR} that offer a complete overview on the topic.

Presumably, part of the difficulty of Mazur rotations problem  lies in the existence of a wide variety of separable, almost transitive (AT for short, see Definition~\ref{def:AT} below) Banach spaces and that
 the difference between transitivity and almost transitivity is so subtle that we cannot rule out the possibility that every AT space could be renormed to be transitive; see \cite[Section~1.5]{CFR}. Actually, the only partial answer to Mazur problem that deserves that name is the Ferenczi-Rosendal result \cite[Theorem 28]{FR2} that a separable transitive space is strictly convex (and smooth, of course) and its dual is AT.

\smallskip

The paper develops two ideas already present in \cite{CDKKLM} and \cite{BRP}. The first one is to 
study the qualitative (read topological) aspects of the action of the isometry group on the unit sphere.

The second is semitransitivity: each {transitivity} property has a {\em halved} version which is obtained by replacing (the group of) isometries by (the semigroup of) contractive automorphisms in the definition. Semitransitivity properties tend to be more stable than their transitive ancestors, see \cite[Remark 2.3]{CDKKLM}, \cite[Lemma 2.3]{BRP} or Theorem~\ref{th:UMST} below.

Many of our results are connected with the 
differentiability of the norm. 
The touchstone of any advance on the isometric part of Mazur's problem
regarding smothness of the norm is the almost trivial fact that a norm $\|\cdot\|$ is Euclidean if and only if the function $x\mapsto{1 \over 2}\|x\|^2$ is twice G\^ateaux differentiable at the origin.
 As for the isomorphic part, we have the highly nontrivial fact that $X$ is isomorphic to a Hilbert space
if, and only if, both $X$ and $X^*$ admit G\^ateaux (but in the end Fr\'echet) differentiable bump functions whose differentials are locally Lipschitz, see \cite[Chapter V, Corollary 3.6]{DGZ} and \cite[Section 3]{fabian}.

Let us explain the organization of the paper and highlight
the main results. 

This section contains, apart from this general introduction, 
the notations and conventions to which we will 
adhere.

Section~\ref{sec:act} introduces the basic definitions used along the paper and shows that the isometry group of $L_p(\mu)$ is discrete in the norm topology for $p\neq2$.

Section~\ref{sec:affairs} contains our main results on microtransitive norms. Let $X$ be a Banach space with unit sphere $S$ and  isometry group $\Iso(X)$. 
Let us say that $X$ (or its norm) is $\Lambda$-Lipschitz-transitive ($\Lambda$-LT for short, $\Lambda\ge 1$) if, given $x,y\in S$, there exists $T\in\Iso(X)$ such that $y=Tx$ {and} $\|T-{\bf I}_X\|\le \Lambda\|y-x\|$.  Here, and throughout the paper, ${\bf I}_X$ denotes the identity on $X$. This means that the action of $\Iso(X)$ on $S$ is not only open but even a ``Lipschitz quotient''.
 
We show that either $1$-LT or its halved version characterize Euclidean norms (a result whose proof I owe to Sergei Ivanov), that $\Lambda$-LT norms are analytic, and that a Banach space which is $\Lambda$-LT for some $\Lambda<3$ is isomorphic to a Hilbert space.

Most of Section~\ref{sec:ST} deals with (almost) transitive norms on Hilbert spaces, 
although it opens with an application of Effros's microtransitivity theorem: if $Y$ is a finite-dimensional 1-complemented subspace of a separable transitive Banach space, then $Y$ is uniformly micro-semitransitive (with the meaning that the semigroup of contractive automorphisms of $Y$ ``acts transitively'' on its unit sphere and the action is uniformly open; see Definition~\ref{def:ST}). Semitransitive renormings of Hilbert spaces are characterized by means of tangent ellipsoids.  A number of examples, some of them related to the results of Section~\ref{sec:affairs}, are included. Finally, 
it is shown that any transitive renorming of a separable Hilbert space is twice G\^ateaux differentible off the origin.

The closing Section~\ref{sec:misc} contains some open problems and miscellaneous remarks main\-ly of a bibliographical nature.

The exposition is a bit terse; the reader can find all the necessary background, the basics of the rotations business, and much more, in the already mentioned papers \cite{becerra, CFR}.
This applies especially to the use of ultraproducts, which will not even be defined here.

Although the paper contains no figures its geometric content is obvious. We strongly encourage the reader to make the pertinent drawings.

Some of the results presented here were obtained long time ago and already appeared in my 1996 thesis \cite{th}, written under the supervision of J.M.F. Castillo (from whom I have plagiarised the title of the paper). These are Proposition~\ref{prop:bad}, Theorem~\ref{th:anal}, and Theorem~\ref{th:AThilbert} and its Corollary. I found them very disappointing at that time; 25 years later I look back on them with more affection and I decided to publish them so that they 
become accessible to anyone
and can live their lives.

\subsection*{Notations, conventions}
We consider linear spaces over the field of real numbers; our results extend straighforwardly to complex spaces: after all, each complex space is also a real one.
An Euclidean norm is one that arises from an inner product in the form $|x|=\sqrt{\langle x|x\rangle}$. We tend to use the notation $|\cdot|$, possibly with subscripts, for Euclidean norms.
A Hilbert space is a Banach space whose norm is Euclidean. A complex space is a Hilbert space if and only the underlying real space is. 
The (unit) sphere of a normed space $X$ is the set $S_X=\{x\in X:\|x\|= 1\}$ and the (closed, unit) ball is $B_X=\{x\in X:\|x\|\le 1\}$.
If no risk of confusion arises we omit the subscripts, in which case we denote by $S^*$ the sphere of the dual space $X^*$. The value of $x^*\in X^*$ at $x\in X$ is denoted by $\langle x^*,x\rangle$.

A (not necessarily linear) mapping $f:X\longrightarrow Y$, acting between metric spaces, is said to be Lipschitz if there is a constant $L$ such that $d(f(x),f(y))\le Ld(x,y)$ for all $x,y\in X$. If this holds for $L=1$ we say that $f$ is contractive, or a contraction.

An operator is a bounded linear map. 
The space of operators from $X$ to $Y$ is denoted by $\mathscr L(X,Y)$ or just  $\mathscr L(X)$ if $Y=X$. The group of linear, surjective isometries of $X$ is denoted by $\Iso(X)$, the semigroup of contractive isomorphisms by $\Aut_1(X)$.

Given a Banach space $X$ and $k\ge 2$, we write $\mathscr L(^k X)$ for the space of all bounded $k$-linear forms $a: X^k\to\R$ with the usual norm $\|a\|=\sup\{|a(x_1,\dots, x_k)|: \|x_i\|\leq 1 \;\forall \;i\leq k\}$. (If $k=1$ we write $X^*$ as usual.) This space is canonically isometric to the dual of the $k$-fold projective tensor product $X\widehat{\otimes}\cdots \widehat{\otimes}X$ which endows it with a weak* topology.
A bilinear form $a:X^2\longrightarrow \R$ is positive-definite if there is a constant $c>0$ such that $a(x,x)\geq c\|x\|^2$ for all $x\in X$.

Regarding differentiability issues we have followed \cite[Chapters 4 and 6]{BL}. Since no norm has directional derivatives at zero let us agree that when speaking about a (Gateaux, Fr\'echet, $C^k$) differentiable or analytic norm we mean that it is so off zero.

The function $x\longmapsto \frac{1}{2}\|x\|^2$ (which is clearly differentiable at the origin) is G\^ateaux differentiable at $x$ if and only if there exists a unique $f\in X^*$ such that $\langle f, x\rangle=\|f\|\|x\|$ and $\|f\|=\|x\|$. This $f$, denoted by $J(x)$, agrees with the G\^ateaux differential of the function $\frac{1}{2}\|\cdot\|^2$ at $x$. If $X$ is smooth (which means that $\frac{1}{2}\|\cdot\|^2$ is everywhere G\^ateaux differentiable or, equivalently,
the norm is G\^ateaux differentiable on the sphere) the map $J:X\longrightarrow X^*$ sending $x$ to $J(x)$ is (correctly defined and it is) called the duality map of $X$.

As usual $f(t)=o(t)$ means that $f(t)/t\longrightarrow 0$ as $t\longrightarrow 0$, while $f(t)=O(t)$ means that there is some $\eps>0$ such that $f(t)/t$ is bounded for $|t|\le\eps$.

Other definitions and notations will be introduced as they are needed.

Abreviations used in the text:
\begin{itemize}
\item AT, almost transitive, Definition~\ref{def:AT}.
\item LT, Lipschitz-transitive, $\Lambda$-LT, Definition~\ref{def:LT}.
\item MT, microtransitive, Definition~\ref{def:MT}.
\item SOT, strong operator topology, the paragraph following Theorem~\ref{th:UMST}.
\item ST, semitransitive, Definition~\ref{def:ST}.
\item UMST, uniformly micro-semitransitive, Definition~\ref{def:ST}.
\end{itemize}

\section{The action of the isometry group}\label{sec:act}
Let us start with the main definitions.

\begin{defi}\label{def:AT}
A normed space $X$ is almost transitive (AT) if given $x,y\in S$ and $\eps>0$ there is $T\in\Iso(X)$ such that $\|y-Tx\|\leq\eps$. If this can be achieved even for $\eps=0$ we say that $X$ is transitive.
\end{defi}

As we already mentioned we shall explore the consequences of very strong forms of transitivity based on the possibility of taking $T$ close to the identity if $y$ is close to $x$:

\begin{defi}\label{def:MT}
A normed space $X$ is microtransitive (MT) if for every $\eps>0$ there is $\delta>0$ such that for every $x,y\in S$ satisfying $\|y-x\|\le\delta$ there is $T\in\Iso(X)$ such that $y=Tx$ and $\|T-{\bf I}_X\|\le\eps$.
\end{defi}

Let $G$ be a group. A homomorphism $\pi$ from $G$ to the group of bijections of $X$ is called an {\em action} of $G$ on $X$. If $G$ and $X$ carry topologies then an action of $G$ on $X$ is continuous if the map $G\times X\longrightarrow X$ given by $(g,x)\longmapsto \pi(g)(x)$ is (jointly) continuous. A (continuous) action is said to be microtransitive if it is transitive and for each $x\in X$ the evaluation map $g\in G\longmapsto \pi(g)(x)\in X$ is open. 
Definition~\ref{def:MT} means that $\Iso(X)$ acts microtransitively on $S$ when $\Iso(X)$ is equipped with the norm topology.

The Effros microtransitivity theorem states that every continuous transitive action of a Polish (separable and completely metrizable) group on a Polish space is microtransitive; see \cite{effros,vanMill}.
We refer the reader to \cite[Chapter I, \S\,9]{kechris} for an introduction to Polish groups as well as the basic examples.

 Unfortunately, the full isometry groups of Banach spaces have a strong tendency to be nonseparable (see Proposition~\ref{prop:bad} below) and so this theorem does not provide MT. Things become much better when one considers the strong operator topology (SOT). 
We shall exploit this fact later, see Theorem~\ref{th:UMST}.

Another situation where Effros's theorem applies is when there is a {\em small} subgroup of $\Iso(X)$ that acts transitively on $S$. To be more specific, let $\mathscr I$ be an ideal of operators on $X$. Assume that $\mathscr I$ is complete under an ideal norm $\|\cdot\|_{\mathscr I}$ that dominates the operator norm. Let us denote by $\Iso_\mathscr I(X)$ the subset of isometries that can be written as $T={\bf I}_X+J$, with $J\in \mathscr I$. 
Then $\Iso_\mathscr I(X)$ is a topological group with the distance inherited from $\mathscr I$. (Note that \cite[Theorem 3.4]{FR} and Pietsch's result in \cite{pietsch} imply that if ${\bf I}_X+J$ is an isometry with respect to any equivalent norm of $X$, then either $J$ is approximable or there is $V\in\mathscr L(X)$ such that ${\bf I}_X+VJ$ has infinite-dimensional kernel.)

Some natural choices of $\mathscr J$ yield Polish groups of isometries: for instance if $X$ has separable dual and $\mathscr I$ is either the ideal of compact operators or the ideal of nuclear operators, then $\Iso_\mathscr I(X)$ is a Polish group. (See \cite[Lemma 2.6]{kania} for a nice proof of the separability of the ideal of compact operators on spaces with separable dual).

 In these circumstances, if $\Iso_\mathscr I(X)$ acts transitively on $S$ then the action is MT for the topology induced by $\|\cdot\|_{\mathscr I}$ and, therefore, $X$ is MT in the sense of Definition~\ref{def:MT}. 
Actually one can prove (and it is implicitly done along the proofs of Lemma 2.5 and Theorem 2.6 in \cite{CDKKLM}) that a separable Banach space $X$ is MT if and only if there exist a Polish subgroup of $(\Iso(X),\|\cdot\|)$ that acts transitively on $S$. 
 We will not insist on this issue. We close this section with a remark on the (bad) behaviour of the isometry groups of most classical spaces.

\begin{prop}\label{prop:bad} Let $p\in[1,\infty)$ different from $2$ and let $\mu$ be an measure. If $T$ is an isometry of $L_p(\mu)$ different from the identity, then  $\|T-{\bf I}\|>\sqrt{2}$. If $\mu$ has no atoms then $T- {\bf I}$ is an isomorphism on some infinite-dimensional subspace of $L_p(\mu)$. 
\end{prop}

\begin{proof}
We may assume that $\mu$ is strictly localizable.
Let $\Sigma$ be algebra of measurable sets underlying $\mu$ in which we identify $A$ and $B$ if $\mu(A\Delta B)=0$, where $A\Delta B=(A\backslash B)\cup (B\backslash A)$ is the symmetric difference of $A$ and $B$. Put $\Sigma_0=\{A\in\Sigma: \mu(A)<\infty\}$.

By the Banach-Lamperti theorem (see \cite[Theorem 0.4]{GJK} for more details when $\mu$ fails to be $\sigma$-finite) there exists a mapping $\Phi:\Sigma_0\longrightarrow\Sigma$ (preserving finite unions and intersections) and a measurable function $h$ such that $T(1_A)=h1_{\Phi(A)}$ for every $A\in\Sigma_0$, where $1_A$ denotes the characteristic function of $A$. Now either 
\begin{itemize}
\item $\mu(A\Delta \Phi(A))=0$ for every $A \in\Sigma_0$, or
\item there exists $A\in\Sigma_0$ such that $\mu(A\Delta \Phi(A))>0$. 
\end{itemize}
In the first case we have $|h|=1$ a.e. If $h=1$ then $T={\bf I}$; otherwise there is $B\in\Sigma_0$ such that $h|_B=-1$ and so $T(1_B)=-1_B \implies \|T-{\bf I}\|\ge 2$.

In the second case we have $\mu(A\backslash \Phi(A))\cup (\Phi(A)\backslash A)>0$ and we can assume 
$\mu(A\backslash \Phi(A))>0$; the case where $\mu(\Phi(A)\backslash A)>0$ is analogous. Put $B=A\backslash \Phi(A)$. Clearly $1_B$ and $T(1_B)$ have disjoint support, hence $\|T(1_B)- 1_B\|=2^{1/p}\|1_B\| \implies \|T-{\bf I}\|\ge 2^{1/p}$.
Since $L_p(\mu)^*$ is isometric to $L_q(\mu)$, where $p^{-1}+q^{-1}=1$, and 
 $\|T-{\bf I}\|= \|T^*-{\bf I}\|$, we get $\|T-{\bf I}\|\geq \max\big(2^{1/p}, 2^{1/q}\big)> \sqrt{2}$ in all cases. The second part is clear since $T-{\bf I}$ is an isomorphism when restricted to $L_p(B)=\{f:\supp f\subset B\}$. 
\end{proof}

The preceding result implies that the isometry groups of all real Lindenstrauss spaces are discrete in the norm topology. A Lindenstrauss space is a Banach space whose dual is isometric to $L_1(\mu)$ for some measure $\mu$; see \cite{CFR} for the main examples of AT Lindenstrauss spaces.

\section{Lipschitz transitivity}\label{sec:affairs}
In this section we prove our main result on MT norms when the action of $\Iso(X)$ on $S$ is a Lipschitz quotient (that is, one has MT with $\delta>c\eps$ for some $c>0$ which corresponds to the inverse of $\Lambda$ in the following definition):

\begin{defi}\label{def:LT}
Let $\Lambda\ge 1$.
A Banach space $X$ is $\Lambda$-Lipschitz-transitive ($\Lambda$-LT) if  for every $x,y\in S$ there is $T\in \Iso(X)$ such that $y=Tx$ and $\|T-{\bf I}_X\|\leq \Lambda\|y-x\|$. We say that $X$ is Lipschitz-transitive (LT) if it is $\Lambda$-Lipschitz-transitive for some $\Lambda\ge 1$.
\end{defi}

The only known LT norms are the Euclidean ones which satisfy the condition with $\Lambda=1$. Actually this property (or its halved version) characterizes Hilbert spaces:

\subsection{An isometric characterization of Hilbert spaces by contractive automorphisms} The following result might fit better within the context of semitransitivity, which is the subject of Section~\ref{sec:ST}. However we will prove it right now because we are aware that it may be the most (only?) interesting part of the paper for some (most?) readers.

\begin{theorem}\label{th:1LT}
For a Banach space $X$ the following are equivalent:
\begin{itemize}
\item[(i)] $X$ is a Hilbert space.
\item[(ii)] $X$ is Lipschitz-transitive with constant $1$.
\item[(iii)] For each $x,y\in S$ there exists a contractive automorphism $T$ of $X$ such that $y=Tx$ and $\|T-{\bf I}_X\|=\|y-x\|$.
\item[(iv)]
There is a contractive mapping $K:S\longrightarrow S^*$  such that $\langle K(x),x\rangle=1$ for every $x\in S$.
\end{itemize}
\end{theorem}

The implication (i) $\implies$ (ii) is obvious in the real case, but not so for complex Hilbert spaces (see \cite[Lemma 2.2]{acosta} for the proof). 

The implication (ii) $\implies$ (iii) is trivial.
(iii) $\implies$ (iv) is consequence of the following remark and the fact that (iii) implies that $X$ is smooth (\cite[Corollary 2.13]{CDKKLM} or  \cite[Proposition 2.6]{BRP}).

\begin{lemma}\label{lem:Lip(J)}
Let $x,y\in S$ and let $T\in\Aut_1(X)$ be such that $y=Tx$. If the norm is smooth at $x$ then so is at $y$ and $\|J(x)-J(y)\|\leq \|T-{\bf I}_X\|$.
\end{lemma}

\begin{proof}
Given 
 $y^*\in S^*$ 
  one has
$
\langle y^* , Tx\rangle = \langle T^*y^* , x\rangle.
$ 
It follows that the norm is smooth at $y$ and  $T^*J(y)=J(x)$. Since $\|T^*-{\bf I}_{X^*}\|= \|T-{\bf I}_{X}\|$ we have
$$
\|J(x)-J(y)\|=\|T^*J(y)-J(y)\|\leq  \|T^*-{\bf I}_{X^*}\|\|J(y)\|= \|T-{\bf I}_{X}\|.\qedhere
$$
\end{proof}
Thus (iii) implies that (iv) holds when $K=J$ is the duality map. The lion's share of the proof of Theorem~\ref{th:1LT} is the implication (iv) $\implies$ (i) whose gorgeous proof, due to Serguei Ivanov, is included here with his kind permission.

The hypotheses imply that for every subspace $Y\subset X$ the composition
$$
\xymatrixcolsep{5.5pc}
\xymatrix{
S_Y\ar[r]^{\text{inclusion}} & S \ar[r]^K & S^* \ar[r]^{\text{restriction to $Y$}} & {Y^*}
}
$$ 
takes values in the sphere of $Y^*$ and has the same properties as $K$. Thus, it suffices to consider the case where $\dim X=2$. 
Elementary considerations show that $X$ is smooth and thus $K=J$ is the duality map.  
Let $E$ be the ellipse of maximal area (John ellipsoid) enclosed by $S$ and let $|\cdot|$ be the corresponding Euclidean norm on $X$.
Let $|\cdot|^*$ be the dual norm on $X^*$
and let 
 $I:(X,|\cdot|)\longrightarrow (X^*,|\cdot|^*)$ be the canonical isometry implemented by the inner product of $X$ so that $\langle I(x), y\rangle =\langle y| x\rangle$ for $x,y\in X$. Clearly,
 $$
 \|x\|\leq |x| =|I(x)|^*\leq\|I(x)\|^*,
 $$
where $\|\cdot\|^*$ is the norm of $X^*$ (that is, $\|f\|^*=\sup\{|\langle f,x\rangle|: \|x\|\le 1\}$.)
Let $\Sigma=\{x\in X: \|x\|=|x|=1\}$. It is clear that $\Sigma$ contains at least two pairs of opposite points and that $J(x)=I(x)\iff
\|I(x)\|^*=\|x\|\iff x$ is proportional to an element of $\Sigma$. 

Assume $x,y\in\Sigma$ are linearly independent and look at the distance between $x$ and $-y$. One has
$$
\|x+y\|\geq\|J(x)+J(y)\|^*=\|Ix+Iy\|^*=\|I(x+y)\|^*\geq \|x+y\|.
$$
Hence the normalized bisector $(x+y)/|x+y|$ belongs to $\Sigma$.
Since $\Sigma$ is closed, it follows that $\Sigma$ agrees with the boundary of $E$.  It's nice, isn't it?

\subsection{Lipschitz-transitive norms}\label{sec:anal}
Back into Mazur's fold we now study $\Lambda$-LT norms for arbitrary $\Lambda$. The main result in this regard is the following.

\begin{theorem}\label{th:anal}
Lipschitz-transitive norms are analytic.
\end{theorem}

The proof requires some preparation.
We first remark that most (all?) differentiability properties are separably determined in the sense that a function is differentiable or analytic if and only if its restriction to any separable subspace is.
On the other hand, the proof of \cite[Theorem~2.6]{CDKKLM} shows that if $X$ is LT then for every separable subspace $Y\subset X$ there is another separable $Z\subset X$ which is LT (with the same constant as $X$) and contains $Y$. Thus, if suffices to prove the theorem for separable Banach spaces.
We begin with the following observation. 
Recall that  $f:X\longrightarrow Y$ is said to be $\alpha$-homogeneous (or homogenenous of degree $\alpha$) if $f(\lambda x)=\lambda^\alpha f(x)$ for every $x\in X$ and $\lambda\ge 0$. Here $\alpha$ can be any real number.

\begin{lemma}\label{lem:k-homog}
Let $f:X\longrightarrow Y$ be an $\alpha$-homogeneous map acting between normed spaces. If the restriction of $f$ to $S_X$ is Lipschitz, then $f$ is Lipschitz on every annulus $\{r\leq \|x\|	\leq R\}$, with $0<r<R<\infty $. If $\alpha=1$ then $f$ is globally Lipschitz.
\end{lemma}

\begin{proof}
Although the statement is almost trivial we write down a complete proof because we need a more precise estimate to complete the proof of Theorem~\ref{th:anal}. Let $L$ be the Lipschitz constant of $f$ on $S_X$ and $M=\sup_{\|x\|=1}\|f(x)\|$. Then if $\|x\|=\|y\|=r$ one has
$$
\|f(x)-f(y)\|\leq r^{\alpha-1}\|x-y\|
$$ 
(obvious)
and $\|f(x)\|\leq M \|x\|^\alpha$ (even more obvious).
For arbitrary $x,y\neq 0$ put $z=\big(\|x\|/\|y\|\big) y$, so that $z$ has the same norm as $x$ and is proportional to $y$. Then 
\begin{equation*}
\|x-z\|\leq 2\frac{\|x\|}{\|y\|}\|x-y\|.
\end{equation*}
(Just estimate $\big\|\|y\|x- \|x\| y\big\|$.) We have $\|f(x)-f(y)\|\leq \|f(x)-f(z)\|+ \|f(z)-f(y)\|$. Also,
\begin{align}\label{eq:using1}
\|f(x)&-f(z)\|\leq \|x\|^{\alpha-1}L \|x-z\|\leq 2L\frac{\|x\|^\alpha}{\|y\|}\|x-y\|;\\ \label{eq:using2}
\|f(z)&-f(y)\|=\left\|\left(\frac{\|x\|}{\|y\|}\right)^\alpha f(y)-f(y)\right\|= \frac{\big|\|x\|^\alpha- \|y\|^\alpha \big|}{\|y\|^\alpha}\|f(y)\|\\\nonumber
&\leq M \big|\|x\|^\alpha- \|y\|^\alpha \big|\leq M \alpha\max\big(\|x\|^{\alpha-1},\|y\|^{\alpha-1}\big)\|x-y\|,
\end{align}
by the mean value theorem. Finally, if $0<r\le\|x\|,\|y\|\le R$ we get
$$
\|f(x)-f(y)\|\leq \left(\frac{2L\max(r^\alpha,R^\alpha)}{r}+
 M \alpha\max\big(r^{\alpha-1},R^{\alpha-1}\big)\right)\|x-y\|,
$$
which proves the first part. If $\alpha=1$, one obtains $\|f(x)-f(z)\|\le 2L\|x-y\|$ and $\|f(z)-f(y)\|\leq M\|x-y\|$  so, certainly, $f$ is Lipschitz.
\end{proof}

We now quote the basic differentiation stuff we need in the form we need:
\smallskip

\noindent$(\dag)\;$Let $X, Y$ be  separable Banach spaces and $g:U\longrightarrow Y^*$ a Lipschitz function, where $U$ is an open subset of $X$. Then $g$ is weak* differentiable on a dense subset of $U$. (See Aronszajn \cite[\S~2, Theorem~1, p. 164]{arons} or \cite[Corollary 6.44]{BL} for a stronger result; in this form the result follows from Mankiewicz \cite[Theorem 4.5]{mank})
\smallskip 

Here, we say that $g$ is weak* differentiable at $x$ if for every $z\in X$ the limit
$$
d_xg(z)=\lim_{t\to 0}\frac{g(x+tz)-g(x)}{t}
$$
exists in the weak* topology of $Y^*$ and is a bounded operator in $z$. This is called a weak*-Gateaux differential in \cite{BL} and indeed it is the Gateaux differential when one considers the norm topology on $X$ and the weak* topology on $Y^*$.
\smallskip

\noindent$(\ddag)\;$Let $X, Y$ be Banach spaces with $X$ separable and let $g:U\longrightarrow Y^*$ a function, where $U$ is an open subset of $X$.
If $g$ is norm-to-weak* continuous, uniformly bounded and weak* differentiable on $U$ and the map $x\in U\longmapsto d_xg\in\mathscr L(X,Y^*)$ is norm continuous and uniformly bounded, then $d_x g$ is a differential in the Fr\'echet sense; Aronszajn \cite[\S~1, Theorem~3, p. 162]{arons}.
\smallskip 

The first portion of the proof of Theorem~\ref{th:anal} appears now: 

\begin{lemma}\label{lem:Coo}
Lipschitz-transitive norms on separable Banach spaces are $C^\infty$-smooth.
\end{lemma}

\begin{proof}
Note that a function which is $C^k$ in the G\^ateaux sense is $C^k$ in the Fr\'echet sense too; see \cite[Proposition 4.2]{BL}. During the proof we consider $\mathscr L(^n X)$ as the dual of the $n$-fold tensor $X\widehat{\otimes}\cdots \widehat{\otimes}X$, mostly for notational reasons.
Let $X$ be separable and LT with constant $\Lambda$. We will prove that the function $f(x)=\|x\|$ is $C^\infty$ on $X\backslash \{0\}$ by showing that, for each $k\geq 0$, it has the following properties:
\begin{itemize}
\item[(a)] $f$ is $k$-times Fr\'echet differentiable on $X\backslash \{0\}$.
\item[(b)] For each nonzero $x\in X$, every $\lambda>0$ and every $T\in\Iso(X)$ one has
$
d_x^k f= (d_{Tx}^kf)(T\otimes\cdots\otimes T)$ and 
$d_x^k f= \lambda^{k-1} d_{\lambda x}^k f.
$
In particular the norm $\|d_x^k f\|$ is the same for all $x\in S_X$ and if we denote it by $M_k$, then  $ \|d_z^k f\|=M_k/\|z\|^{k-1}$ for every nonzero $z\in X$.
\item[(c)]  The map $d^kf: S_X \longrightarrow \mathscr L(^k X)$ is bounded and, given $x,y\in S_X$ one has $\|d_y^{k}f- d_x^{k}f\|\leq k M_k \Lambda\|x-y\|$, where $M_k$ is as in (b). 
\end{itemize}
When $k\geq 1$ and $d_x^kf$ is interpreted as a $k$-linear form on $X$ the first equality in (b) means  $d_x^k f(y_1,\dots,y_k)= (d_{Tx}^kf)(Ty_1,\dots,Ty_k)$ for $y_i\in X$, while the second one means that $d^k_xf$ is  homogenenous of degree $1-k$ in $x$.

The proof is by induction on $k$.
The case $k=0$ is trivial, although it is more entertaining than one might expect:
The 0-fold tensor product is the ground field and
$d^0_xf$ is the constant $\|x\|$ acting by multiplication on $\R$, while the ``0-fold tensor power'' of $T$ is ${\bf I}_\R$. 
The identities in (b) say that isometries preserve the norm (!) and that the norm is homogeneous (!!), respectively. 
Finally (c) means that the norm is constant on the unit sphere (!!!).
Some readers may prefer to start the induction with $k=1$ (using Lemma~\ref{lem:Lip(J)} and taking into account that that differential of the norm agrees with the duality map on the sphere) avoiding this exegesis of the integer $0$.

Let us assume that (a), (b), (c) hold for $k=n$ and check the corresponding statements for $k=n+1$. 

We first observe that (c), the second part of (b) and the preceding lemma imply that prove that the map $x\in X\backslash \{0\}\longmapsto d_x^n f\in \mathscr L(^n X)$ 
is locally Lipschitz. 
Since $\mathscr L(^n X)$ is the dual of $X\widehat{\otimes}\cdots \widehat{\otimes}X$, which is separable, $(\dag)$ implies that
$d^nf: X\backslash\{0\}\longrightarrow \mathscr L(^n X)$
 is weak* differentiable on a dense subset of $X$. Of course this implies that $d^nf$ is weak* differentiable away from $0$. Indeed, assume 
 $x\neq 0$ is a point of weak* differentiability of $d^nf$ and let $T\in \Iso(X)$ and $\lambda>0$. The identities of (b) can be written as
$$
d^n f= (T\otimes\cdots\otimes T)^*\circ d^nf\circ T,\qquad
 \lambda^{n-1} d^n f\circ{\lambda{\bf I}_X}
$$
which are perhaps best understood looking at the commutative diagrams
$$
\xymatrixcolsep{5pc}
\xymatrix{
X \ar[d]_{d^nf} \ar[r]^T & X \ar[d]^{d^nf}\\
\mathscr L(^n X)  & \mathscr L(^n X) \ar[l]_-{(T\otimes\cdots\otimes T)^*}
}\qquad\quad
\xymatrixcolsep{5pc}
\xymatrix{
X \ar[d]_{d^nf} \ar[r]^{\text{multiplication by $\lambda$}} & X \ar[d]^{d^nf}\\
\mathscr L(^n X)  & \mathscr L(^n X) \ar[l]_{\lambda^{n-1}}
}
$$
Since the upper arrows are automorphisms of $X$ and the lower ones are weak* automorphisms of $\mathscr L(^n X)$ we see that $d^nf$ is weak* differentiable at $x$ if and only if it is at $Tx$ or $\lambda x$. But  the action of $\Iso(X)$ is transitive on spheres and so $d^nf$ is weak* differentiable at every $x\neq 0$.

Moreover, by the chain rule,
\begin{align*}
d_x d^n f &= (T\otimes\cdots\otimes T)^*\circ d_{Tx}d^nf\circ T,\\
 d_x d^n f&=  \lambda^{n-1}d_{\lambda x}d^nf\circ\lambda=  \lambda^{n}d_{\lambda x}d^nf,
\end{align*}
which,
switching to the multilinear mode through the canonical isometries
$$
\mathscr L\big(X, (\underbrace{X\widehat{\otimes}\cdots \widehat{\otimes}X}_{\text{$n$-times}})^*\big)=\mathscr L\big(X, \mathscr L(^n X)\big)= \mathscr L(^{n+1} X),
$$
 means that for every $y,y_1,\dots,y_n\in X$ one has
\begin{align*}
 d^{n+1}_x f(y,y_1,\dots,y_n)&= d^{n+1}_{Tx} f(Ty,Ty_1,\dots,Ty_n),\\
d^{n+1}_{x} f(y,y_1,\dots,y_n)&= \lambda^{n}d^{n+1}_{\lambda x} f(y,y_1,\dots,y_n).
\end{align*}
Thus the weak* differential $d_xd^n f=d_x^{n+1}f$ exists at every $x\neq 0$ and one has (b) for $k=n+1$. It remains to see that $d_xd^n f$ are actually Fr\'echet differentials and that (c) holds for $k=n+1$. But if (c) holds then the second part of (b) and the previous lemma imply that $x\longmapsto d_x d^nf= d_x^{n+1}f$ is Lipschitz on the annuli $r\le\|x\|\le R$ and ($\ddag$) guarantees that $d_xd^n f$ are Fr\'echet differentials, so it suffices to check (c). This is the point where LT is crucial:
Pick $x,y\in S_X$ and let $T\in \Iso(X)$ be such that $y=Tx$, with $\|T-{\bf I}_X\|\leq \Lambda\|x-y\|$. Then
\begin{align*}
\|d_y^{n+1}f- d_x^{n+1}f\|&= 
\|d_y^{n+1}f- d_{Tx}^{n+1}f (T\otimes\cdots\otimes T)\|\\
&\leq 
\|d_y^{n+1}\|\|{\bf I}_{X\widehat\otimes\cdots\widehat\otimes X}- T\otimes\cdots\otimes T\|\\
&\leq (n+1) \|d_y^{n+1}f\|\|{\bf I}_X-T\|\\
&\leq (n+1) M_{n+1} \Lambda\|x-y\|.\qedhere
\end{align*}
\end{proof}

\begin{proof}[End of the proof of Theorem~\ref{th:anal}]
We must prove that for each non\-zero $x\in X$ the Taylor series
\begin{equation}\label{eq:taylor}
\sum_{n=0}^\infty \frac{d_x^nf}{n!} (\underbrace{h,\dots,h}_{\text{$n$ times}})
\end{equation}
($f$ is the norm on $X$) converges uniformly to $\|x+h\|$ for $\|h\|$ sufficiently small. We may, and do, assume that $\|x\|=1$. Our immediate aim is to obtain an estimate of the norms $\|d_x^n f\|$, i.e., of the sequence $(M_n)_{n\geq 1}$. 

Obviously $M_1= 1$. The point is that the local Lipschitz constants of $d^n f: X\backslash\{0\}\longrightarrow \mathscr L(^n X)$ can be somehow controlled by $M_n$. This yields a bound of $\|d^{n+1}_x f\|$.

Recall from the proof of Lemma~\ref{lem:Coo} that if $\|x\|=\|y\|=1$, then
$$
\|d_y^{n}f- d_x^{n}f\| \leq n M_n \Lambda\|x-y\|.
$$
Let us estimate $\|d_x^{n+1}f\|=M_{n+1}$. 
Note that $d_x^{n+1}f= d_x d^n f$ can be seen as an operator
$X\longrightarrow \mathscr L(^n X)$, so that
$$
\|d_x^{n+1}f\|=\sup_{\|y\|\leq 1} \|d_x d^{n}f(y)\|_{\mathscr L(^n X)}
$$
Pick $y\in X$. We have
$$
d_x d^{n}f(y) =\lim_{t\to 0}\frac{d^n_{x+ty}f- d^n_{x}f}{t}.
$$
Using the inequalities (\ref{eq:using1}) and (\ref{eq:using2}) of the proof of Lemma~\ref{lem:k-homog} (with $x+ty$ instead of $y$ and $\alpha=1-n$) and taking into account that $\|x\|=1$ we obtain
$$
\big\|{d^n_{x+ty}f- d^n_{x}f}\big\|\leq 
2nM_n\Lambda \frac{\|ty\|}{\|x+ty\|}
+
M_n|n-1|\max\big(1,  \|x+ty\|^{-n}\big)\|ty\|.
$$
Dividing by $t$
and letting $t\longrightarrow 0$ we obtain $
\| d_x d^{n}f(y) \|\leq M_n\big((n-1)+2n \Lambda \big)\|y\|
$, so
$$ M_{n+1}\leq 3n\Lambda M_n\\
\quad \implies\quad  M_n\leq (n-1)!(3\Lambda )^{n-1}.
$$
Finally, to prove that the series (\ref{eq:taylor}) converges to $\|x+h\|$ for $\|h\|$ sufficiently small we can use the integral form for the remainder
$$
r_n(h)=f(x+h)-\sum_{k=0}^n \frac{d_x^kf}{k!}(\underbrace{h,\dots,h}_{\text{$k$ times}}) = \frac{1}{n!}\int_0^1(1-t)^n d_{x+th}^{n+1} f(\underbrace{h,\dots,h}_{\text{$n\!+\!1$ times}})\, dt,
$$
see for instance Cartan \cite[Th\'eor\`eme 5.6.1]{cartan}.

We know that the norm of $d_{x+th}^{n+1} f$ depends only on $\|x+th\|$ and since $\|x+th\|\geq 1-t\|h\|$ we have
$$
\|d_{x+
th}^{n+1} f\|\leq \frac{M_{n+1}}{\big(1-t\|h\|\big)^n}.
$$
Hence, for $\|h\|\leq 1$,
\begin{align*}
|r_n(h)| &=   \frac{1}{n!}\int_0^1(1-t)^n d_{x+th}^{n+1} f(h,\dots
, h)dt\\
&\leq  \frac{1}{n!}\int_0^1(1-t)^n \| d_{x+th}^{n+1} f\| \|h\|^{n+1}dt\\
& \leq 
\frac{M_{n+1}\|h\|^{n+1}}{n!}\int_0^1\frac{(1-t)^n} {\big(1-t\|h\|\big)^n} dt\\
&\leq 
\frac{n!(3\Lambda)^n\|h\|^{n+1}}{n!}\int_0^1\left(\frac{1-t} {1-t\|h\|}\right)^n dt\\
&\leq (3\Lambda)^n\|h\|^{n+1}.
\end{align*}
Thus the Taylor series (\ref{eq:taylor}) converges uniformly to $\|x+h\|$ on any ball of radius $r<1/(3\Lambda)$.
\end{proof}

The preceding result  drastically reduces  the list of Banach spaces that could have a LT norm. Indeed MT Banach spaces are uniformly convex and Deville's {\em Great Theorem} (see \cite[Chapter V, \S 4]{DGZ} for an exposition) implies that a space $X$ containing no copy of $c_0$ with a $C^\infty$-smooth norm has the following properties:
\begin{enumerate}
\item $X$ has type 2 (due to Fabian, Whitfield, and Zizler);
\item $X$ has exact cotype $q$, where $q$ is an even integer;
\item $X$ contains an isomorphic copy of $\ell_q$, where $q$ is the cotype of $X$.
\item There is a $q$-form $a$ on $X$ and constants $C, c>0$ such that $c\|x\|^q\leq a(x)\le C\|x\|^q$ for all $x\in X$.
\end{enumerate}
Consequently, a LT space is isomorphic to a Hilbert space in any of the following cases:
\begin{itemize}
\item $X^*$ is isomorphic to $X$;
\item $X^*$ has a LT renorming;
\item $X$ does not contain $\ell_q$ for $q\geq 4$, even.
\end{itemize}
In fact, we can go a little further: according to 
\cite[13.16 Proposition]{MS} every infinite-dimensional $K$-convex space (something that all super-reflexive spaces are) with exact cotype $q$ contains uniformly complemented, almost isometric copies of $\ell_q^n$ (precisely there is $C$ such that for every $n\in \N$ and every $\eps>0$ the space contains a subspace whose Banach-Mazur distance to $\ell_q^n$ is at most $1+\eps$ which is complemented by a projection of norm at most $C$; this is a ``complemented'' version of the Maurey-Pisier theorem and it actually follows from the results in \cite{pisier}).
Thus if $X$ is LT and not isomorphic to a Hilbert space, then
$X$ has exact cotype $q\ge 4$, even, and it is $K$-convex. Hence,
if $\mathscr U$ is a free ultrafilter on $\N$, the ultrapower $X_\mathscr U$ contains a complemented subspace $Y$ isometric to $\ell_q$ (and another one isometric to $L_q$, by the way). But $X_\mathscr U$ is LT with the same constant than $X$ and by the comment after Theorem~\ref{th:anal} there is a further separable LT subspace $Z$ such that $Y\subset Z\subset X_\mathscr U$. We, therefore, have the following no-result:

\begin{noresult}\label{no:either}
Either every Lipschitz-transitive space is isomorphic to a Hilbert space or there is one which is separable and contains a complemented subspace isometric to $\ell_q$ for some $q\ge 4$, even.
\end{noresult}

If the LT constant is small we do much better:

\begin{cor}
If $X$ is Lipschitz transitive with constant $\Lambda<3$, then $X$ is isomorphic to a Hilbert space.
\end{cor}

\begin{proof}
By the preceding argument, if $X$ is not isomorphic to a Hilbert space, then taking an ultrapower we may assume that $X$ contains an isometric copy of $\ell_q^2$ for some $q\ge 4$ even.

Let $I:\ell_q^2\longrightarrow X$ be an isometric embedding and $I^*:X^*\longrightarrow \ell_p$ be the dual quotient map, where $p^{-1}+q^{-1}=1$. We know from Lemma~\ref{lem:Lip(J)} that the duality map $J:S\longrightarrow S^*$ is $\Lambda$-Lipschitz. It is clear that the composition
$$
\xymatrixcolsep{3.5pc}
\xymatrix{
S_{\ell_q^2}\ar[r]^I & S \ar[r]^J & S^* \ar[r]^{I^*} & {\ell_p^2}
}
$$ 
takes values in the sphere of $\ell_p^2$ since no step can increase the norm and $\langle I^*JI(y), y\rangle= \langle JI(y), Iy\rangle= 1$ provided $\|y\|_q=1$. This implies that $I^*JI$ is the duality map on the sphere of $\ell_q^2$ which agrees with the Mazur map given by
$M_{qp}(x,y)= \big(\sigma(x)|x|^{q/p}, \sigma(y)|x|^{q/p}\big)$, where $\sigma$ is the signum function. The following result shows that this map cannot be $\Lambda$-Lipschitz for any $\Lambda<3$ if $q\ge 4$ and ends the proof.
\end{proof}

\begin{lemma}
Let $1<p<q<\infty$ be arbitrary (not necessarily conjugate) exponents. Then the Lipschitz constant of the Mazur map $M_{qp}:\ell_q^2\longrightarrow\ell_p^2$ on the unit sphere of $\ell_q^2$ is bounded below by $q/p$.
\end{lemma}

\begin{proof}
This is surely well-known (cf. \cite[Proof of Theorem 9.1]{BL}); we include the proof for the sake of completeness. 
We write $M$ instead of $M_{qp}$ and work on the first quadrant 
so that $M(x,y)= (x^{q/p},y^{q/p})$. We have
\begin{align*}
\|(x,y)-(y,x)\|_q &=  \|(x-y, y-x)\|_q = 2^{1/q}|x-y|;\\
\|M(x,y)-M(y,x)\|_p &=  \|(x^{q/p}\!-y^{q/p}, y^{q/p}\!-x^{q/p})\|_p= 2^{1/p}|x^{q/p}\!-y^{q/p}|.
\end{align*}
Thus, if $x^q+y^q=1$,
$$
\frac{\|M(x,y)-M(y,x)\|_p}{\|(x,y)-(y,x)\|_q}
= \frac{2^{1/p}}{2^{1/q}} \frac{|x^{q/p}-y^{q/p}|}{ |x-y|}
= \frac{2^{1/p}}{2^{1/q}} \bigg{|} \frac{x^{q/p}-(1-x^q)^{1/p}}{ x- (1-x^q)^{1/q}}\bigg{|}.
$$
Letting $s=x^q$ and applying L'H\^opital's rule we see that the limit of the function inside the absolute value as
 $x\longrightarrow 2^{-1/q}$ 
is 
$$
 \lim_{s\to 1/2} \frac{s^{1/p}-(1-s)^{1/p}}{s^{1/q}-(1-s)^{1/q}}
 =  \lim_{s\to 1/2} \frac{\frac{1}{p}s^{1/p-1}+ \frac{1}{p}(1-s)^{1/p-1}}{\frac{1}{q}s^{1/q-1}+ \frac{1}{q}(1-s)^{1/q-1}}= \frac{q}{p}
 \frac{2^{-1/p}}{2^{-1/q}}.\qedhere
$$
\end{proof}

Since ultrapowers preserve MT and since no MT space can contain a 1-complemented subspace isometric to $\ell_q^2$ for $q\neq 2$ (this was observed in \cite{CDKKLM}) it is clear that if $q\neq 2$ no MT space can contain for each $\eps>0$ a $(1+\eps)$-complemented subspace ($1+\eps$)-isometric to $\ell_q^2$. Thus, in view of the dichotomy in Assertion 
\ref{no:either}, the following general question arises:

\begin{quest}
Does every separable Banach space containing a complemented subspace isometric to $\ell_q$  contain  $(1+\eps)$-complemented copies of $\ell_q^2$ for all $\eps>0$?
\end{quest}

(While it is relatively easy to see that the answer is affirmative when the copy of $\ell_q$ has finite codimension, it is likely that the answer is negative in general.)

Another consequence is that if $X$ is separable and LT and $\Iso(X)$ is amenable in the SOT (something that the group of isometries of any Hilbert space {\em is}, proved by Gromov and Milman), then the norm of $X$ arises from a $q$-form in the sense that $\|x\|=a(x)^{1/q}$, where $q$ is an even integer that agrees with the cotype of $X$, so the preceding question has some interest even for spaces whose norm comes from a $q$-norm ($q=4,6,\dots$).


Anyway, the hottest question on LT is, of course, if it passes to the dual in which case each LT space would be isomorphic to a Hilbert space. It is perhaps a little ironic that the answer is affirmative when one {\em knows} that the underlying space is isomorphic to a Hilbert space (in particular, if $\Lambda<3$), as it clearly follows from Theorem~\ref{th:AThilbert}; see the proof of Corollary~\ref{cor:2Gateaux}(c).

As in other transitivity problems the main difficulty seems to be that, while the properties of the action of the isometry group on the sphere can be translated into smoothnes properties of the norm (e.g., type, smoothness, modulus of smoothness), we do not have such direct control of its convexity (cotype, strict convexity, modulus of convexity); see the proof of  \cite[Theorem 28]{FR2} for another manifestation of this phenomenon.

\section{Spaces with a large semigroup of contractive automorphisms}\label{sec:ST}

Most of this section deals with renormings of Hilbert spaces.
Our first result, however, applies Effros's theorem to winkle out some information on the (mostly hypothetical) finite-dimensional 1-complemented subspaces of a separable transitive Banach space.

\subsection{Semitransitivity}
Recall that $\Aut_1(X)$ denotes the semigroup of contractive automorphisms of $X$. 

\begin{defi}\label{def:ST}
A normed space $X$ will be called semitransitive (ST) if, given $x,y\in S$, there is $T\in \Aut_1(X)$ such that $y=Tx$. We say that $X$ is uniformly micro-semitransitive (UMST) if for every 
$\eps>0$ there is $\delta>0$ such that for every $x,y\in S$ satisfying $\|x-y\|\le \delta$ there exists $T\in \Aut_1(X)$ such that $y=Tx$ and $\|T-{\bf I}_X\|\le \eps$.
\end{defi}
These are the halved versions of transitivity and MT, respectively; see Section~\ref{sec:MST} to see why we have chosen the second name.
The implications ``transitive $\implies$ ST'' and ``MT $\implies$ UMST'' are obvious. The later property was introduced in \cite[Definition 2.2]{CDKKLM} where it is shown that UMST $\implies$ ST and that UMST  is a self-dual property that is inherited by 1-complemented subspaces (in particular each 1-complemented subspace of a MT space is UMST) and implies both uniform convexity and uniform smoothness. Becerra and Rodr\'\i guez-Palacios show in \cite{BRP} (whose title I have plagiarised for the heading of this section) that all these implications hold even for their $(\eps,	\delta)$-property 
(which is the weak version of UMST obtained by replacing ``for every $\eps>0$'' by ``for some $0<\eps<1$'' and shall not be treated here).


\begin{theorem}\label{th:UMST}
Every $1$-complemented, finite-dimensional subspace of a separable transitive Banach space is uniformly micro-semitransitive.
\end{theorem}

This would be one giant leap for mankind if we knew that transitive spaces contained 1-complemented subspaces of finite dimension greater than 1. Since we do not know if this the case 
the theorem is just a small step, hopefully in the right direction. 

The proof is a true-blue  
application of the Effros microtransitivity theorem.
 Given Banach spaces $X$ and $Y$, the  
strong operator topology (SOT) on  $\mathscr L(X,Y)$ is just
the restriction to product topology of $Y^X$.
This topology makes $\Iso(X)$ into a Polish group when $X$ is separable. A typical neighbourhood of ${\bf I}_X$ in the SOT has the form $$\{T\in\Iso(X): \|x_i-Tx_i\|\leq \varepsilon\, \text{ for all $1\le i\le k$}\},$$ 
where $x_i\in S_X$ and $\eps>0$.

\begin{proof}[Proof of Theorem~\ref{th:UMST}]
Assume $X$ is complete, separable, and transitive and let $Y$ be a finite-dimensional subspace of $X$ complemented by a contractive projection $P$. 

Fix $\eps\in(0,1)$. Let $y_1,\dots, y_k$ be a normalized basis for $Y$ and take $\eta>0$ such that if a contraction $L\in\mathscr L(Y)$ satisfies $\|y_i-Ly_i\|\le \eta$ for $1\leq i\leq k$, then $\|L-{\bf I}_Y\|\le\eps$. Put
$$
V_\eta=\{T\in \Iso(X): \|y_i-Ly_i\|\le \eta\, \text{ for all $1\le i\le k$}\}.
$$
Note that $T\in V_\eta\iff  T^{-1}\in V_\eta$ and that $T\in V_\eta\implies \|T_P- {\bf I}_Y\|\le\eps$, where $T_P=PT|_Y$. Now, given $x\in S_Y$, we define $\delta(x,\eps)$ as the supremum of those numbers $\delta\ge 0$ for which the following condition holds:

\smallskip

$\bullet$
For every $y\in S_Y$ such that $\|x-y\|\le\delta$ there exist $L,R\in\Aut_1(Y)$ such that $\|L-{\bf I}_Y\|, \|R-{\bf I}_Y\|\le\eps, y=Lx, x=Ry$.
\smallskip

By Effros's theorem for every $x\in S_X$ the set $\{Tx: T\in V_\eta\}$ is open in $S_X$ and this clearly implies that $\delta(x,\eps)>0$ for all $x\in S_Y$ and all $\eps\in(0,1)$. It remains to see that  $\delta(x,\eps)\ge c$ for some $c>0$ and all $x\in S_Y$.

If we assume the contrary there is a sequence $(x_n)$ in $S_Y$ such that $\delta(x_n,\eps)\longrightarrow 0$ and so there is another sequence $(y_n)_{n\ge 1}$ such that $\|x_n-y_n\|\longrightarrow 0$ but either there is no $L\in\Aut_1(Y)$ such that $\|L-{\bf I}_X\|\le\eps$ and $y=Lx$ or 
 there is no $R\in\Aut_1(Y)$ such that $\|R-{\bf I}_X\|\le\eps$ and $x=Ry$. 
 
 Interchanging $x_n$ by $y_n$ when necessary we may assume that the first option is always the case. Using the compactness of $S_Y$ and passing to a subsequence we may assume that $(x_n)$ and so $(y_n)$ converge to, say $x$. But for this $x$ we have $\delta(x,\eps/2)>0$ and so $\|x-x_n\|,\|x-y_n\|\le \delta(x,\eps/2)$ for $n$ sufficiently large. For these $n$ there exist $R_n\in \Aut_1(Y)$ such that $\|R_n-{\bf I}_X\|\le\eps/2$ and $x=R_nx_n$ and $L_n\in \Aut_1(Y)$ such that $\|L_n-{\bf I}_X\|\le\eps/2$ and $y_n=L_nx$ and considering $L=L_nR_n$ we obviously reach a contradiction.
\end{proof}

\begin{cor}
Every $1$-complemented, finite-dimensional subspace of a separable transitive Banach space has modulus of convexity and smoothness of power type $2$.
\end{cor}

This is a pleasant consequence of the facts, proved in the UN paper \cite[Corollary 2.14 and Proposition 3.4]{CDKKLM}, that a USMT space isomorphic to a Hilbert space has modulus of convexity of power type 2 and that USMT is a self-dual property. Thus $\ell_p^n$ cannot be the range of a contractive projection on a separable transitive Banach space in striking contrast with 
Lusky's result that {\em every} separable (in particular finite-dimensional) Banach space is 1-complemented in some separable AT Banach space; see \cite[Theorem 2.14]{becerra} for some extensions. An easy consequence is that every separable reflexive (in particular finite-dimensional)  space is 1-comple\-mented in a transitive Banach space whose density character is $\aleph_1$, irrespectively of the ``value'' of the first uncountable cardinal.

The first precursor of the preceding corollary of which I am aware is
the following result of Greim, Jamison and Kami\'nska \cite[Theorem 3.3]{GJK}: separable transitive Banach spaces have trivial $L^p$-structure for $p\neq 2$.

\subsection{Semitransitive norms on Hilbert spaces}
This section studies ST norms on Hilbert spaces. Our motivation for doing so is twofold: first, most features of USMT spaces are actually shared by ST spaces (at least in finite dimensions) and second, unlike general AT spaces, AT renormings of Hilbert spaces are ST.

A subset of a Banach space $X$ which is the image of the unit ball of a Hilbert space $H$ under some linear isomorphism 
$U:H\longrightarrow X$ is called an {\em ellipsoid}. We say that an ellipsoid $E$ is inner (respectively, outer) at $x\in S_X$ if $x\in E$ and $E\subset B_X$ (respectively, $B_X\subset E$).
Of course these notions make sense only for spaces that are isomorphic to (that is, renormings of) Hilbert spaces.
The following duality argument will be used over and over without further mention: if $x\in X$ and $x^*\in X^*$ are such that $\|x\|=\|x^*\|= 1$ and $\langle x^*,x\rangle =1$, then $E$ is inner (respectively, outer) at $x$  if and only if $E^*=\{y^*\in X^*: |\langle y^*, y\rangle |\leq 1\; \forall\, y\in E\}$ is outer (respectively, inner) at $x^*$.

\begin{lemma}\label{lem:delicate}
The unit sphere of any renorming of a Hilbert space has a dense set of points that admit inner ellipsoids as well as some point that admits an outer ellipsoid.
\end{lemma}

\begin{proof}
The second part is a consequence of the first one which in turn can be proved in several ways. The proof that best suits our purposes is based on the following {\em delicacy} of Fabian. A function $\phi:X\longrightarrow\R$ is said to be Lipschitz-smooth at $x$ if there is $x^*\in X^*$ such that $\phi(x+y)-\phi(x)-\langle x^*, y\rangle =O(\|y\|^2)$. Note that the definition forces $x^*=d_x\phi$ in the Fr\'echet sense. Fabian proved in \cite[Theorem 2.8]{fabian} (a result implying) that every continuous convex function on a Hilbert space is Lipschitz-smooth on some dense subset. Therefore the set of Lipschitz-smooth points of the convex function $\|\cdot\|^2$ is dense in the unit sphere of $X$. Let us see that all those points admit inner ellipsoids.
Assume $x\in S$ is such that
\begin{equation}\label{eq:mmm}
\|x+y\|^2=1 +2\langle J(x), y\rangle+O(\|y\|^2).
\end{equation}
Let $|\cdot|$ be an equivalent, Euclidean norm on $X$. 
Given $0<b<\infty$, we define another Euclidean norm on $X$ as follows: if $z=tx+h$, with $t$ real and $h\in\ker J(x)$ we put
$$
|z|_b= \sqrt{t^2+ b^2|h|^2}.
$$
Let $E_b$ be the unit ball of $|\cdot|_b$; note that the larger $b$ is the smaller $E_b$ is and that $x$ belongs to the boundary of each $E_b$.
We claim that $\|\cdot\|\leq |\cdot|_b$ for sufficiently large $b$. Otherwise for each $n\in\N$ we can find $t_n\in \R$ and $h_n\in \ker J(x)$ such that $\|t_nx+h_n\|^2>t_n^2+n^2|h_n|^2$, which implies that $t_n\neq 0$ for large $n$. Dividing by $t_n$ we obtain a sequence $y_n\in \ker J(x)$ such that
$$
\|x+y_n\|^2>1+n^2|y_n|^2\implies y_n\longrightarrow 0,
$$
which seriously compromises (\ref{eq:mmm}).

We cannot help mentioning the following  more  direct proof for the second part: 
Aron, Finet and Werner proved in \cite[Theorem~1]{AFW} that the set of bilinear forms $a:X\times X\longrightarrow\R$ that attain their norms (there are $x_0,y_0\in B_X$ such that $|a(x,y)|\leq |a(x_0,y_0)|$ for $x,y\in B_X$) are dense in the space of all bilinear forms if $X$ has the Radon-Nikod\'ym property---something that all reflexive spaces {\em do}. It is very easy to see that if $a$ is symmetric and positive-definite and attains the norm, then it attains the norm on the diagonal. It quickly follows that if $X$ is isomorphic to a Hilbert space then there is an equivalent Euclidean norm $|\cdot|$ that attains its maximum at some $x\in B_X$ and we can assume that that maximum is 1. Clearly the ellipsoid $E=\{z\in X:|z|\leq 1\}$ is outer at $x$.
\end{proof}

\begin{prop}\label{prop:Aut1}
A renorming of a Hilbert space is semitransitive 
if and only if every point of the unit sphere admits both inner and outer ellipsoids. 
\end{prop}

\begin{proof}
Pick $x,y\in S$. Let $E$ be an outer ellipsoid at $x$ and let $F$ be inner at $y$. If $T$ is an automorphism of $X$ mapping $E$ onto $F$ and such that $y=Tx$ (which exists because Hilbert spaces are transitive and isomorphic Hilbert spaces are isometric) it is clear that $T(B_X)\subset B_X$, that is, $T$ is contractive.

For the converse, by the preceding lemma there is some $x\in S$ admiting an outer ellipsoid $E$ and some $z\in S$ admitting an inner ellipsoid $F$. Given $y\in S$ let $T,L\in\Aut_1(X)$ such that $x=Ty, y=Lz$. Then  $T^{-1}(E)$ is outer at $y$ and $L(F)$ is inner at $x$.
\end{proof}

An obvious consequence of the proposition just proved  (and already remarked in \cite{CDKKLM} for UMST) is that $\ell_p^n$ cannot be ST unless $n=1$ or $p=2$: indeed if $1\le p<2$ the unit vectors do not admit inner ellipsoids; if $2<p\le \infty$ they do not admit outer ellipsoids. Thus, while arbitrary subspaces or quotients of ST renormings of Hilbert spaces are ST, 1-complemented subspaces of transitive spaces are not necessarily ST.

 This raises the question, implicit in \cite{BRP} and explicit in \cite{CDKKLM} (cf. the comment after Remark~2.4), of the existence of (non-Euclidean, finite-dimen\-sional or separable) ST spaces.
En route to the concrete examples let us give a simple sufficient condition for a point to admit inner and outer ellipsoids:

\begin{lemma}
Assume $X$ is isomorphic to a Hilbert space and let $\varphi(x)={1\over 2}\|x\|^2$. If $\varphi$ is twice Fr\'echet differentiable at $x\in S_X$ then $x$ admits inner ellipsoids. If besides $d^2_x\varphi$ is positive-definite, then $x$ admits outer ellipsoids.
\end{lemma}

\begin{proof}
The first part follows from the proof of Lemma~\ref{lem:delicate} just because ``twice Fr\'echet differentiable at $x$'' implies ``Lipschitz-smooth at $x$''.

Each symmetric, positive-definite bilinear form on $X$ defines (is) an inner product whose associated Euclidean norm is equivalent to the original norm.
The hypotheses yield
$$
\tfrac{1}{2}\|x+h\|^2= \tfrac{1}{2}+\langle J(x),h\rangle + \tfrac{1}{2}d^2_x\varphi(h,h)+ o\big(\|h\|^2\big).
$$
Let $|\cdot|$ be the Euclidean norm associated to  $\tfrac{1}{2}d^2_x\varphi$, that is, $|y|^2=\tfrac{1}{2}d^2_x\varphi(y,y)$. It is every easy to see that $\ker J(x)$ is orthogonal to $x$ and that $J(x)=\tfrac{1}{2}d_x|\cdot|^2$. It follows that
$$
\tfrac{1}{2}|x+h|^2= \tfrac{1}{2}+\langle J(x),h\rangle+ \tfrac{1}{2}d^2_x\varphi(h,h)
$$
(exactly!) for all $h\in X$. Given $0<b<\infty$, we define an Euclidean norm on $X$ as follows: if $z=tx+h$, with $t$ real and $h\in\ker J(x)$ we put
$$
|z|_b= |tx+bh|= \sqrt{t^2+ b^2|h|^2}.
$$
Let $E_b$ be the unit ball of $|\cdot|_b$. Since $|\cdot|$ is equivalent to $\|\cdot\|$ it is really easy to see that $B_X\subset E_b$ if $b$ is sufficiently small (otherwise $d^2_x\varphi$ should not be positive-definite).
\end{proof}

\subsection{Examples of analytic norms whose dual norm is analytic}\label{sec:ex}
Each continuous $\pi$-periodic function $g:\R\longrightarrow (0,\infty)$ defines a continuous {\em quasinorm} 
on $\R^2$ by the formula
$$
|v|_g={r}/{g(\theta)},
$$
where $(r,\theta)$ are the polar coordinates of $v$. (Don't worry if you don't know what quasinorms are.)
Such a $|\cdot|_g$ will be a norm if and only the set $B_g=\{v\in\R^2: |v|_g\leq 1\}= \{v\in\R^2: r\leq g(\theta)\}$ (the unit ball of $|\cdot|_g$)  is convex.
A simple, sufficient condition for the (strict) convexity of $B_g$ is that $g$ is twice differentiable and satisfies the inequality (a) in the next lemma. Note that $|\cdot|_g$ is Euclidean if and only if $g$ has the form
$$
g(\theta)=\frac{b}{\sqrt{1-e^2\cos^2(\theta-\theta_0)}}
$$
for some $b>0$ and $0<e\le 1$ (remember when you were taught curves in polar coordinates).
Now, we have:

\begin{prop} Let $g:\R\longrightarrow (0,\infty)$ be an analytic function of period $\pi$ such that:
\begin{itemize}
\item[(a)] $(g')^2+gg''<g$;
\item[(b)] $g''g-4(g')^2+g^2>0$.
\end{itemize}
Then $|\cdot|_g:\R^2\longrightarrow\R$ is an analytic norm and so is the dual norm.
\end{prop}

\begin{proof}
To prove that $|\cdot|_g$ is a norm it suffices to see that the region enclosed by the curve $S_g=\{v\in\R^2: r=g(\theta), \theta\in\R\}$ is convex.

If we fix $\theta$, the equation of the tangent line to $S_g$ at the point whose polar coordinates are $(g(\theta), \theta)$ is 
\begin{equation*}
r_{\text{tan}}(\theta+ \Delta\theta)=\sqrt{g(\theta)^2+2\,\Delta\theta\, g(\theta)\,g'(\theta)+  2(\Delta\theta\, g'(\theta))^2}
\end{equation*}
for $|\Delta\theta|<\pi/2$, while the equation of $S_g$ can be written, using the second-degree Taylor polynomial of $g$ at $\theta$, as
\begin{equation*}
r= g(\theta+ \Delta\theta)=g(\theta) + g'(\theta) \Delta\theta+ \frac{g''(\theta)}{2}(\Delta\theta)^2+  o\big( (\Delta\theta)^2 \big)
\end{equation*}
If (a) holds, a simple inspection reveals that $g(\theta+ \Delta\theta)\le r_{\text{tan}}(\theta+ \Delta\theta)$ for all $\theta$ if $|\Delta\theta|$ is sufficiently small from where the convexity of $B_g$ follows.

Once we know that $|\cdot|_g$ is a norm, it is clear that it is real analytic off 0 and so is the duality map $J={1\over 2}d|\cdot|_g^2$. As the reader may guess the proof of the analytic character of the dual norm $|\cdot|_g^*$ is indirect.

Applying twice the chain rule one easily obtains
$$
\tfrac{1}{2}d_P^2|\cdot|_g^2(J(u),J(v))
= \tfrac{1}{2} u^{\text{tr}}\underbrace{\left(\begin{matrix}
\frac{\partial^2}{\partial r^2}\frac{r^2}{g(\theta)^2} & \left(\frac{\partial^2}{\partial r \partial \theta} -\frac{1}{r}\frac{\partial f}{\partial \theta}\right)\frac{r^2}{g(\theta)^2}
\\
\text{\small{symmetric}} &\left(\frac{\partial^2}{\partial \theta^2}+r \frac{\partial f}{\partial r}\right)\frac{r^2}{g(\theta)^2}
\end{matrix}\right)}_{H} v,
$$
where $P,u,v\in\R^2$ and $u^{\text{tr}}$ is the transpose of $u$ and the partial derivatives are evaluated at $P$ (or at the polar coordinates of $P$, depending on the interpretation one prefers); see Massicot \cite{mass} for a crystal clear explanation. Differentiating,
$$
 H=  2 \left(\begin{matrix}
\frac{1}{g^2} & \frac{-rg'}{g^3}
\\
 \frac{-rg'}{g^3} & {r^2}\left(\frac{g''}{g^3}-\frac{3(g')^2}{g^4}\right)+\frac{r^2}{g^2}
\end{matrix}\right).
$$
Hence
$$\det  H
=
\frac{4}{g^4}\left|\begin{matrix}
1 &   \frac{-rg'}{g}
\\
 \frac{-rg'}{g}  & r^2\left( \frac{g''}{g} - \frac{3(g')^2}{g^2}+1 	\right)
\end{matrix}\right| = 
\frac{4 r^2}{g^4}\left(  \frac{g''}{g} - \frac{4(g')^2}{g^2}+1	\right)
$$
which is strictly positive if (b) holds.

The inexorable conclusion is that $dJ$ is non-singular at any $P\neq 0$ and the inverse function theorem for analytic functions shows that $J^{-1}$ is analytic off the origin. But $J^{-1}$ is the differential of ${1\over 2}(|\cdot|^*_g)^2$ and so $|\cdot|_g^*$ is real analytic off the origin.
\end{proof}

The proof actually shows that conditions (a) and (b) together characterize those norms on the real plane which are analytic and have analytic dual norm. 
Note that any slowly varying $g$ such that $g(0)=1$ satisfies the hypotheses of the proposition, for instance
$
g(\theta) = 1+\epsilon\sin^k(n\theta)
$
for arbitrary $k,n\in\mathbb N$ and $\epsilon>0$ small enough. Needless to say, the corresponding norms are ST but not Euclidean. Proceeding more carefully one can construct non-Euclidean norms which satisfy the $(\eps,\delta)$-property; see \cite[Definition 2.1]{BRP} if you are in the mood for it. It is unclear, however, if this simple procedure can give USMT norms since our {technique} to check ST is a bit palaeolithic.

\subsection{(Almost) Transitive norms on Hilbert spaces}
This section deals with the isometric part of Mazur's problem. 
Our main result is:

\begin{theorem}\label{th:AThilbert}
Let $X$ be an almost transitive renorming of a Hilbert space. Then the duality map is (correctly defined and) a Lipschitz equivalence between $X$ and $X^*$.
\end{theorem}

To put the theorem in perspective observe that if $X$ is super-reflexive and AT, then it is uniformly convex and uniformly smooth and $X^*$ is AT in the dual norm. It follows that $J$ is a uniform homeomorphism between the spheres of $X$ and $X^*$ which in  general does not extend to a uniform homeomorphism between $X$ and $X^*$ (just consider $X=L_p$ for some $p\in(1,\infty)$ different from 2). Any Banach space which satisfies the conclusion of the theorem is isomorphic to a Hilbert space, as it follows from the already mentioned \cite[Chapter V, Corollary 3.6]{DGZ}. 


Throughout the section $X$ will denote a Banach space isomorphic to a Hilbert space. We denote by $\|\cdot\|$ the norm of $X$ and by $B$ and $S$ the unit ball and unit sphere of $\|\cdot\|$, respectively.

\begin{proof}[Proof of Theorem~\ref{th:AThilbert}]
Let $x\in S$ be a point of Lipschitz-smoothness of the convex function $\tfrac{1}{2}\|\cdot\|^2$. Then there are $C,\delta>0$ such that if  for every $y\in S$ with $\|y-x\|\leq\delta$ one has
\begin{equation*}
1- \langle J(x), y\rangle\leq |\langle J(x), x-y\rangle|\leq C\|x-y\|^2.
\end{equation*}
Using the AT character of the norm and the continuity of $J$ we see that the preceding inequality holds (with the same $C,\delta$) for every $x\in S$. On the other hand, if we apply the same result to $X^*$, then \cite[Proposition 2.2]{fabian} tell us that $B$ must be Lipschitz exposed at some $x\in S$, necessarily by $J(x)$. This means that for some $c>0$ and all $y\in B$ one has
$
c\|x-y\|^2\leq 1- \langle J(x), y\rangle
$. Applying the same argument as before and relabelling the constants we conclude that there exist $k,K,\delta>0$ such that, for $x,y\in S$,
$$
\|x-y\|\leq\delta \implies k(1- \langle J(x), y\rangle )\leq \|x-y\|^2\le  K(1- \langle J(x), y\rangle).
$$
Since $X^*$ is also an AT renorming of a Hilbert space there exist in particular
$K^*,\delta^*>0$ such that, for $x^*, y^*\in S^*$,
$$
\|x^*-y^*\|\leq\delta^* \implies 
\|x^*-y^*\|^2\leq  K^*(1- \langle J^{-1}(x^*), y^*\rangle).
$$
Recalling that  $J:S\longrightarrow S^*$ is uniformly continuous we can choose $\eps\leq \delta$ such that  $x,y\in S, \|x-y\|\leq \eps\implies \|J(x)-J(y)\|\leq \delta^*$. Now, for  $\|x-y\|\leq \eps$,   taking $x^*=J(x), y^*=J(y)$ we get
$$
\|y^*-x^*\|^2\leq K^* (1- \langle J^{-1}(y^*), x^*\rangle)=
 K^* (1- \langle  J(x), y\rangle)
\leq \frac{K^*}{k}\|x-y\|^2.
$$
If follows that $J$ is Lipschitz on $S$ and the proof concludes just applying the last part of Lemma~\ref{lem:k-homog}.
\end{proof}

Time for applications.

\begin{cor}\label{cor:2Gateaux}
Let $X$ be a renorming of a Hilbert space.
\begin{itemize}
\item[(a)] If $X$ is separable and transitive, the norms of $X$ and $X^*$ are twice G\^ateaux differentiable.
\item[(b)] If $X$ is micro-transitive, the norms of $X$ and $X^*$ are $C^2$-smooth.
\item[(c)] If $X$ is Lipschitz-transitive, the norms of $X$ and $X^*$ are analytic.
\end{itemize}
\end{cor}

\begin{proof}
(a) is clear from Theorem~\ref{th:AThilbert}. (b) One can assume that $X$ is separable; cf. the paragraph between Theorem~\ref{th:anal} and Lemma~\ref{lem:k-homog}. Reasoning as in the proof of Lemma~\ref{lem:Coo} it is very easy to see that the Gateaux differential $x\mapsto d^2_x\|\cdot\|\in \mathscr L(^2 X)$ is  bounded and (uniformly) continuous on $S$, hence locally bounded and continuous on $X\backslash\{0\}$. According to ($\ddagger$) this implies that $d^2_x\|\cdot\|$ are actually Fr\'echet differentials. (c) In view of Theorem~\ref{th:anal} it suffices to see that $X^*$ is LT too. Which is obvious from Theorem~\ref{th:AThilbert}: let $\Lambda$ be the LT constant of $X$ and let $L$ be the Lipschitz constant of $J^{-1}:S^*\longrightarrow S$. Pick $x^*,y^*\in S^*$ and let $x=J^{-1}(x), y=J^{-1}(y^*)$, so that $\|x-y\|\leq L\|x^*-y^*\|$. Take $T\in\Iso(X)$ such that $x=Ty$ and $\|T-{\bf I}_X\|\leq \Lambda\|x-y\|$. Then $y^*=T^*x^*$ and $\|T^*-{\bf I}_{X^*}\|=\|T-{\bf I}_X\|\leq \Lambda\|x-y\|\leq \Lambda  L\|x^*-y^*\|$.
\end{proof}

One can prove that if $Y$ is a finite-dimensional 1-complemented subspace of a LT space, then $Y$ {\em and} $Y^*$ have analytic norms. 
Thus, in spite of our efforts, some of the (two-dimensional) spaces
appearing at the end of Section~\ref{sec:ex} might be 1-complemented in a LT space which could even be a renorming of a Hilbert space.

We close with the following remark that can be used to get a simpler,  longer, geometric proof of Theorem~\ref{th:AThilbert} and has
some intrinsic interest.
 
\begin{prop}\label{prop:inner/outer}
Every point of the unit sphere of an almost transitive renorming of a Hilbert space admits inner and outer ellipsoids.
\end{prop}

\begin{proof}
(The proof is an ultraproduct argument in disguise.)
We prove the existence of outer ellipsoids; the existence of the inner ones follows then by duality.
We know from the proof of Proposition~\ref{prop:Aut1} that there exist outer ellipsoids at some points of the sphere so the main difficulty lies in the {\em almost} assumption.
Assume $E$ is outer at $y\in S$ and
let $|\cdot|$ be the corresponding Euclidean norm. We have $|y|=\|y\|= 1$ and $|z|\leq\|z\|\leq M|z|$ and all $z\in X$.

Let $x\in S$ and for each $n$ take $T_n\in \Iso(X)$ such that $\|y-T_nx\|\leq 1/n$. Let $\mathscr U$ be a free ultrafilter on $\mathbb N$ and define a norm on $X$ by the formula
$$
|z|_\mathscr U= \lim_{\mathscr U(n)}|T_n z|.
$$
Then we have $
|\cdot|_\mathscr U\leq \|\cdot\|\leq M |\cdot|_\mathscr U$ and $|\cdot|_\mathscr U$ is homogenenous and satisfies the  parallelogram law so it is an Euclidean norm. Moreover,
$$
|x|_\mathscr U= \lim_{\mathscr U(n)}|T_n x|=\left| \lim_{n\to\infty} T_n x\right|=|y|=1.\qedhere
$$ 
\end{proof}
Thus, while AT renormings of Hilbert spaces are ST, there exist AT spaces that are not ST (consider $L_1$ or the Gurariy space).
We do not know whether $L_p$ is ST for some $p\in(0,\infty)$ different from $1$ and $2$.

\subsection{UMST vs. ``MST'', ``LST''}\label{sec:MST}
The attentive reader may have noticed that the halved relative of MT has been called UMST (instead of the more logical micro-semitransitivity, MST). There are two reasons for doing so: First, to be consistent with \cite{CDKKLM}, where UMST appears for the first time. Second, one can imagine another condition that could be labelled MST, this one: for every $x\in S$ and every $\eps>0$ there is $\delta=\delta(x,\eps)>0$ such that, whenever $y\in S$ and $\|y-x\|\le\delta$ there exists $T\in \Aut_1(X)$ such that $y=Tx$ and $\|T-{\bf I}_X\|\le\eps$. This property is strictly weaker than UMST (and even than ST) and indeed $\ell_1^2$ (which obviously is not ST, let alone UMST) has it.

We have not considered ``Lipschitz-semitransitivity'' (LST) because we have not much to say about, with the only exception of $1$-LST (which is (iii) in Theorem~\ref{th:1LT}). It is clear from Lemmas~\ref{lem:Lip(J)} and \ref{lem:k-homog} that if $X$ is LST, then the duality map $J:X\longrightarrow X^*$ is globally Lipschitz; hence if $X^*$ is also LST (not necessarily in the dual norm) then $X$ is isomorphic to a Hilbert space.

\section{Questions, proper names, and acknowledgements}\label{sec:misc}

\subsection*{Some more food for thought}
The content of this paper naturally flows into the following questions:

\begin{quest} Is every microtransitive Banach space isometric (or isomorphic) to a Hilbert space?
\end{quest}

Even for this cheap version of Mazur problem the
gap between an eventual affirmative answer and the existing knowledge is oceanic: on the one hand we do not even know what happens with LT spaces or norms and, on the other, we know of no UMST space whose norm is not Euclidean. As for concrete spaces one may ask:

\begin{quest}Let $q\neq 2$.
Does $L_q$ admit a microtransitive renorming? 
Does $\ell_q$ admit a uniformly semi-micro\-tran\-si\-tive norm? What if  $q$ is even?
\end{quest}

But note that no result in this paper can rule out 
the possibility that $\ell_4\oplus\ell_2$ has  a $\Lambda$-LT renorming for some $\Lambda\ge 3$!

Regarding the second part of the above question, we would like to draw to the reader's attention that Dilworth and Randrianantoanina, using the asymptotic structure of the spaces $\ell_q$ with masterly skill, have managed to prove that these do not have an AT renorming unless $q=2$; see \cite[Theorem 2.4]{DR}.

\subsection*{Wojty\'nski}
The notion of a LT norm was introduced in \cite{wojt}, where it is claimed that every  LT norm on a separable Banach space is Euclidean. The proof, however, has a fatal error in the very last line. 
This seems to have been detected in \cite[Note 1 on p. 152]{odinec}. The proof of Theorem~\ref{th:anal} resumes, and takes to its extreme, the healthy part of Wojty\'nski's proof.

\subsection*{Finet}
When \cite{th} was written I didn't know who Fabian was and the basic result of Aron, Finet and Werner in \cite{AFW} was not yet available (or at least I didn't know it).
The fact that AT renormings of Hilbert spaces have inner and outer ellipsoids at every point of the unit sphere, used \emph{there} in the proof of Theorem~\ref{th:AThilbert}, 
was deduced from a nice result of Finet {\em alone} (namely that if a Banach space admits a norm with modulus of convexity of power type $p$ then all its AT renormings have modulus of convexity of power type $p$; see \cite{finet} or \cite[Corollaries 5.6 and 5.7]{DGZ} and \cite[Theorem 1.2]{CS3} for a simplification of the proof)
throught the following fact: if $X$ is a renorming of a Hilbert space then $x\in S$ admits an outer ellipsoid if and only if the (Lovaglia, local) modulus of convexity of $X$ at $x$ is of power type 2.

\subsection*{Becerra \& Rodr\'\i guez-Palacios} The sphere of any possible counterexample to the isometric part of Mazur rotations problem would be a fascinating infinite-dimensional Riemannian Hilbert manifold. Considering now complex scalars, it is remarkable that, while  such a sphere should be a homogeneous manifold, its ``interior'', that is, the open ball $D$ should be a very rigid {\em domain} in which all biholomorphic automorphisms leave the origin fixed. (Note that the biholomorphic automorphisms act transitively on the open  balls of Hilbert spaces.)

This is a consequence of the work of Knaup, Upmeier and Vigu\'e, who proved the
 astonishing fact that each complex Banach space $X$ carries inside a closed subspace $X_{\text{sym}}$, called its symmetric part, having the following properties:
 \begin{itemize}
 \item The orbit of the origin under the action of the group of biholomorphic automorphisms of $D$, $\Hol(D)$, is  $X_{\text{sym}}\cap D=D_{\text{sym}}$, the unit ball of $X_{\text{sym}}$.
 
\item $\Hol(D_{\text{sym}})$ acts transitively on $D_{\text{sym}}$.
 \end{itemize}
This is all finely  explained  by Arazy in \cite{arazy}, see   \cite[Section 5.6]{CRP} for a more complete exposition. Being clear that $\Iso(X)$ leaves $X_{\text{sym}}$ invariant, it follows that if $X$ is AT (much weaker conditions suffice) then either $X_{\text{sym}}=0$ (and then the elements of $\Hol(D)$ are restrictions of those of $\Iso(X)$) or $X_{\text{sym}}=X$ and, therefore, $X$ is a JB*-triple.

Tarasov proved in \cite{tarasov} that the only smooth JB*-triples are the Hilbert spaces.
To the best of my knowledge, apart from this result, the most serious study of the connections between Mazur's problem and the JB*-triples was performed by Becerra and Rodr\'\i guez-Palacios in \cite{BP3}; see also \cite[Section~5]{becerra}, \cite[Section~6]{BRP}, and \cite{BRPW}.

\subsection*{Ivanov}
Although MT seems to be an overwhelmingly restrictive notion, the embarrasing truth is that I didn't even know how to prove the implication (ii) $\implies$ (i) of Theorem~\ref{th:1LT} until I saw Ivanov's argument on the internet and realized that his proof not only worked when $K$ is isometric (as stated in \cite{ivanov}) but also when $K$ is contractive (as is actually used in the proof of ``our'' result). Needless to say, I'm most grateful to Prof. Ivanov both for the proof itself and for the permission to include it here. I must add that the undisputable fact that Theorem~\ref{th:1LT} is the nicest result of the paper makes me feel a bit like the soloist in Brahms' violin concerto, whose most beautiful part is played by the oboe.

\end{document}